\documentclass[12pt,a4paper]{amsart}

\usepackage{fullpage,amsxtra,amssymb,eucal,amsthm,amsmath,url}
\usepackage[utf8]{inputenc}

\numberwithin{equation}{section}
\renewcommand{\leq}{\leqslant}
\renewcommand{\geq}{\geqslant}

\def\stacksum#1#2{{\stackrel{{\scriptstyle #1}}
{{\scriptstyle #2}}}}

\newcommand{\Cc}{\mathbf{C}}

\newcommand{\Zz}{\mathbf{Z}}

\newcommand{\Aa}{\mathbf{A}}

\newcommand{\Rr}{\mathbf{R}}
\newcommand{\Gg}{\mathbf{G}}

\newcommand{\Qq}{\mathbf{Q}}
\newcommand{\Fp}{\mathbf{F}}

\newcommand{\Oc}{\mathcal{O}}

\newcommand{\mods}[1]{\,(\mathrm{mod}\,{#1})}

\DeclareMathSymbol{\dgenb}{\mathord}{symbols}{"1F}
\DeclareMathSymbol{\dgena}{\mathord}{symbols}{"1E}
\DeclareMathSymbol{\dgenaa}{\mathord}{symbols}{"1C}
\DeclareMathSymbol{\dgenbb}{\mathord}{symbols}{"1D}

\newcommand{\ra}{\rightarrow}
\newcommand{\lra}{\longrightarrow}

\newcommand{\injecte}{\hookrightarrow}

\newcommand{\fleche}[1]{\stackrel{#1}{\lra}}

\newcommand{\charfun}{\mathbf{1}}

\newcommand{\eps}{\varepsilon}
\renewcommand{\rho}{\varrho}

\newcommand{\sifted}{\mathcal{S}}

\newcommand{\rd}[1]{\ar@{{*}-{*}}[r]_-{{#1}}}
\newcommand{\lrd}[1]{\ar@{{*}-{*}}[rr]_-{{#1}}}

\newcommand{\fb}[1]{\ar@{{*}->}[r]^-{{#1}}}
\newcommand{\lfb}[1]{\ar@{{*}->}[rr]^-{{#1}}}
\newcommand{\me}[1]{\ar@{{*}~{*}}[r]^-{{#1}}}
\newcommand{\sub}{\ \ar@{>->}[r]}

\DeclareMathSymbol{\gena}{\mathord}{letters}{"3C}
\DeclareMathSymbol{\genb}{\mathord}{letters}{"3E}

\def\multsum{\mathop{\sum\cdots \sum}\limits}

\DeclareMathOperator{\Gal}{Gal}

\DeclareMathOperator{\proba}{\mathbf{P}} 
\DeclareMathOperator{\expect}{\mathbf{E}}    
\DeclareMathOperator{\variance}{\mathbf{V}}    
\DeclareMathOperator{\SL}{SL}
\DeclareMathOperator{\GL}{GL}

\DeclareMathOperator{\Sp}{Sp}

\theoremstyle{plain}
\newtheorem{theorem}{Theorem}[section]

\newtheorem{lemma}[theorem]{Lemma}
\newtheorem{corollary}[theorem]{Corollary}

\newtheorem{proposition}[theorem]{Proposition}

\theoremstyle{remark}
\newtheorem{remark}[theorem]{Remark}

\theoremstyle{definition}

\newtheorem{assumption}[theorem]{Assumption}
\newtheorem{definition}[theorem]{Definition}

\newtheorem{example}[theorem]{Example}

\newcounter{exercices}
\newcounter{etape}

\begin{document}

\title{Sieve in discrete groups, especially sparse}
 
\author{Emmanuel  Kowalski}
\address{ETH Z\"urich -- D-MATH\\
  R\"amistrasse 101\\
  8092 Z\"urich\\
  Switzerland} 
\email{kowalski@math.ethz.ch}

\date{}

%%\subjclass[2010]{20F69, 05C50, 05C81} 


\begin{abstract}
  We survey the recent applications and developments of sieve methods
  related to discrete groups, especially in the case of infinite
  index subgroups of arithmetic groups.
\end{abstract}

\maketitle

%%\tableofcontents

\section{Introduction}\label{sec-intro}

Sieve methods appeared in number theory as a tool to try to understand
the additive properties of prime numbers, and then evolved over the
20th Century into very sophisticated tools. Not only did they provide
extremely strong results concerning the problems most directly
relevant to their origin (such as Goldbach's conjecture, the Twin
Primes conjecture, or the problem of the existence of infinitely many
primes of the form $n^2+1$), but they also became tools of crucial
important in the solution of many problems which were not so obviously
related (examples are the first proof of the Erd\"os-Kac theorem, and
more recently sieve appeared in the progress, and solution, of the
Quantum Unique Ergodicity conjecture of Rudnick and Sarnak).
\par
It is only quite recently that sieve methods have been applied to new
problems, often obviously related to the historical roots of sieve,
which involve complicated infinite discrete groups (of exponential
growth) as basic substrate instead of the usual integers. Moreover,
both ``small'' and ``large'' sieves turn out to be applicable in this
context to a wide variety of very appealing questions, some of which
are rather surprising. We will attempt to present this story in this
survey, following the mini-course at the ``Thin groups and
super-strong-approximation'' workshop. The basic outline is the
following: in Section~\ref{sec-framework}, we present a sieve
framework that is general enough to describe both the classical
examples and those involving discrete groups; in
Section~\ref{sec-implement}, we show how to implement a sieve, with
emphasis on ``small'' sieves. In Section~\ref{sec-large}, we take up
the ``large'' sieve, which we discuss in a fair amount of details
since it is only briefly mentioned in~\cite{bourbaki} and has the
potential to be a very useful general tool even outside of
number-theoretic contexts. Finally, we conclude with a sampling of
problems and further questions in Section~\ref{sec-pbs}. 
\par
We include a general version of the Erdös-Kac Theorem in the context
of affine sieve (Theorem~\ref{th-erdos-kac}), which follows easily
from the method of Granville and Soundararajan~\cite{granville-sound}
(it generalizes a result of Djankovi\'c~\cite{djankovic} for
Apollonian circle packings.)
\par
Apart from this, the writing will follow fairly closely the notes for
the course at MSRI, and in particular there will be relatively few
details and no attempts at the greatest known generality.  The final
section had no parallel in the actual lectures, for reasons of time.
More information can be gathered from the author's Bourbaki
lecture~\cite{bourbaki}, or from Salehi-Golsefidy's paper in these
Proceedings~\cite{salehi-golsefidy}, and of course from the original
papers. Overall, we have tried to emphasize general principles and
some specific applications, rather than to repeat the more
comprehensive survey of known results found in~\cite{bourbaki}.
\par
\medskip
\par
\textbf{Notation.} We recall here some basic notation.
\par
\noindent -- The letters $p$ will always refer to a prime number; for
a prime $p$, we write $\Fp_p$ for the finite field $\Zz/p\Zz$. For a
set $X$, $|X|$ is its cardinality, a non-negative integer or
$+\infty$.
% \par
% \noindent: We write $\sumb{}$ to indicate that a sum is restricted
% to integers which are squarefree.
\par
\noindent -- The Landau and Vinogradov notation $f=O(g)$ and $f\ll g$
are synonymous, and $f(x)=O(g(x))$ for all $x\in D$ means that there
exists an ``implied'' constant $C\geq 0$ (which may be a function of
other parameters) such that $|f(x)|\leq Cg(x)$ for all $x\in D$.  This
definition \emph{differs} from that of N. Bourbaki~\cite[Chap. V]{FVR}
since the latter is of topological nature. We write $f\asymp g$ if
$f\ll g$ and $g\ll f$. On the other hand, the notation $f(x)\sim g(x)$
and $f=o(g)$ are used with the asymptotic meaning of loc. cit.
\par
\medskip
\par
\textbf{Reference.} As a general reference on sieve in general, the
best book available today is the masterful work of Friedlander and
Iwaniec~\cite{FI}. Concerning the large sieve, the author's
book~\cite{cup} contains very general results.  We also recommend
Sarnak's lectures on the affine sieve~\cite{sarnak-lecture}.  Another
survey of sieve in discrete groups, with a particular emphasis on
small sieves, is the Bourbaki seminar of the author~\cite{bourbaki},
and Salehi-Golsefidy's paper~\cite{salehi-golsefidy} in these
Proceedings gives an account of the most general version of the affine
sieve, due to him and Sarnak~\cite{sg-s}. 
\par
\medskip
\par
\textbf{Acknowledgments.} Thanks to A. Kontorovich and
A. Salehi-Golsefidy for pointing out imprecisions and minor mistakes
in the first drafts of this paper, and to H. Oh for suggestions that
helped clarify the text.

\section{The setting for sieve in discrete groups}\label{sec-framework}

Sieve methods attempt to obtain estimates on the size of sets
constructed using local-global and inclusion-exclusion principles. We
start by describing a fairly general framework for this type of
questions, tailored to applications to discrete groups (there are also
other settings of great interest, e.g., concerning the distribution of
Frobenius conjugacy classes related to families of algebraic varieties
over finite fields, see~\cite[Ch. 8]{cup}).
\par
We will consider a group $\Gamma$, viewed as a discrete group, which
will usually be finitely generated, and which is given either as a
subgroup $\Gamma\subset \GL_r(\Zz)$ for some $r\geq 1$, or more
generally is given with a homomorphism
$$
\phi\,:\, \Gamma\lra \GL_r(\Zz),
$$
which may not be injective (and of course is typically not
surjective).  Here are three examples.

\begin{example}
(1) We can take $\Gamma=\Zz$, embedded in $\GL_2(\Zz)$ for instance,
using the map
$$
\phi(n)=\begin{pmatrix}1 & n\\0&1
\end{pmatrix}.
$$
\par
This case is of course the most classical.
\par
(2) Consider a finite symmetric set $S\subset \SL_r(\Zz)$, and let
$\Gamma=\langle S\rangle\subset \GL_r(\Zz)$. Of particular interest
for us is the case when $\Gamma$ is ``large'' in the sense that it is
Zariski-dense in $\SL_r$. Recall that this means that there exist no
polynomial relations among all elements $g\in \Gamma$ except for those
which are consequence of the equation $\det(g)=1$. A concrete example
is as follows: for $k\geq 1$, let
$$
S_k=\Bigl\{
\begin{pmatrix}1&\pm k\\0&1
\end{pmatrix}
,
\begin{pmatrix}1&0\\\pm k&1
\end{pmatrix}
\Bigr\}
$$
and let $\Gamma^{(k)}$ be the subgroup of $\SL_2(\Zz)$ generated by
$S_k$. It is well-known that for $k\geq 1$, this is a Zariski-dense
subgroup of $\SL_2$.
\par
We are especially interested in situations where $\Gamma$ is
nevertheless ``small'', in the sense that the index of $\Gamma$ in the
arithmetic group $\SL_r(\Zz)$ is \emph{infinite}. We will call this
the \emph{sparse} case (though the terminology \emph{thin} is also
commonly used, we will wish to speak later of thin subsets of $\SL_r$,
as defined by Serre, and $\Gamma$ is not thin in this sense).
\par
In the example above, the groups $\Gamma^{(1)}=\SL_2(\Zz)$ and
$\Gamma^{(2)}$ are of finite index in $\SL_2(\Zz)$ (the latter is the
kernel of the reduction map modulo $2$), but $\Gamma^{(k)}$ is sparse
for all $k\geq 3$. In particular, the subgroup $\Gamma^{(3)}$ is
sometimes known as the Lubotzky group.
\par
(3) Here is an example where the group $\Gamma$ is not given as a
subgroup of a linear group: for an integer $g\geq 1$, let $\Gamma$ be
the mapping class group of a closed surface $\Sigma_g$ of genus $g$,
and let
$$
\phi\,:\, \Gamma\lra \Sp_{2g}(\Zz)\subset \GL_{2g}(\Zz)
$$
be the map giving the action of $\Gamma$ on the first homology group
$H_1(\Sigma_g, \Zz)\simeq \Zz^{2g}$, which is symplectic with respect
to the intersection pairing on $H_1(\Sigma_g, \Zz)$. Here it is known
(for instance, through the use of specific generators of $\Gamma$
mapping to elementary matrices in $\Sp_{2g}(\Zz)$) that $\phi$ is
surjective. (All facts on mapping class groups that we will use are
fairly elementary and are contained in the book of Farb and
Margalit~\cite{farb-margalit}.)
\end{example}

The next piece of data are surjective maps
$$
\pi_p\,:\, \Gamma\lra \Gamma_p
$$
where $p$ runs over prime numbers (or possibly over a subset of them)
and $\Gamma_p$ are finite groups. We view each such map as giving
``local'' information at the prime $p$, typically by reduction modulo
$p$. Indeed, in all cases in this text, the homomorphism $\pi_p$ is
the composition
$$
\Gamma\fleche{\phi} \GL_r(\Zz)\lra \GL_r(\Fp_p)
$$
of $\phi$ with the reduction map of matrices modulo $p$, and
$\Gamma_p$ is defined as the image of this map.

\begin{example}
(1) For $\Gamma=\Zz$, reduction modulo $p$ is surjective onto
$\Gamma_p=\Zz/p\Zz$ for all primes.
\par
(2) If $\Gamma$ is Zariski-dense in $\SL_r$, and we use reduction
modulo $p$ to define $\pi_p$, it is a consequence of general
\emph{strong approximation} statements that there exists a finite set
of primes $T(\Gamma)$ such that $\pi_p$ has image \emph{equal} to
$\SL_r(\Fp_p)$ for all $p\notin T(\Gamma)$, and in particular for all
primes large enough.\footnote{\ This is directly related to the fact
  that $\SL_r$ is, as a linear algebraic group, connected and simply
  connected.} For instance, in the case of the subgroups
$\Gamma^{(k)}\subset \SL_2(\Zz)$, this property is visibly valid with
$$
T(\Gamma^{(k)})=\{\text{primes $p$ dividing $k$}\}.
$$
\par
We refer to the survey~\cite{rapinchuk} by Rapinchuk in these
Proceedings for a general account of Strong Approximation.
\par
(3) For the mapping class group $\Gamma$ of $\Sigma_g$, and $\phi$
given by the action on homology, the image of reduction modulo $p$ is
equal to $\Sp_{2g}(\Fp_p)$ for all primes $p$ (simply because $\phi$
is onto, and $\Sp_{2g}(\Zz)$ surjects to $\Sp_{2g}(\Fp_p)$ for all
$p$). 
\end{example}

We want to combine the maps $\pi_p$, corresponding to local
information, modulo many primes in order to get ``global''
results. This clearly only makes sense if using more than a single
prime leads to an increase of information. Intuitively, this is the
case when the reduction maps $\pi_p$, $\pi_q$, associated to distinct
primes $p$ and $q$ are \emph{independent}: knowing the reduction
modulo $p$ of an element of $\Gamma$ should give no information
concerning the reduction modulo $q$. We therefore make the following
assumption on the data:

\begin{assumption}[Independence]\label{hyp-indep}
  There exists a finite set of primes $T_1(\Gamma)$, sometimes called
  the \emph{$\Gamma$-exceptional primes}, such that for any finite set
  $I$ of primes $p\notin T_1(\Gamma)$, the simultaneous reduction map
$$
\pi_I\,:\, \Gamma\lra \prod_{p\in I}
{
\Gamma_p
}
$$
modulo primes in $I$ is onto. 
\end{assumption}

We will write 
$$
\Gamma_I=\prod_{p\in I}\Gamma_p,\quad\quad 
q_I=\prod_{p\in I}p.
$$
\par
Note that $q_I$ is a squarefree integer, coprime with $T_1(\Gamma)$.

\begin{example}
(1) For $\Gamma=\Zz$, the Chinese Remainder Theorem shows that for any
finite set of primes $I$, we have
$$
\prod_{p\in I}\Zz/p_i\Zz\simeq \Zz/q_I\Zz,
$$
and hence the map $\pi_I$ above can be identified with reduction
modulo $q_I$. In particular, it is surjective, so that the assumption
holds with an empty set of exceptional primes.
\par
(2) If $\Gamma\subset \GL_r(\Zz)$ has Zariski closure $\SL_r$, then
the Independence Assumption holds for the same set of primes
$T_1(\Gamma)=T(\Gamma)$ such that $\pi_p$ is surjective onto
$\SL_r(\Fp_p)$ for $p\notin T(\Gamma)$, simply for group-theoretic
reasons: any subgroup of a finite product
$$
\prod_{p\in I} \SL_r(\Fp_p)
$$
which surjects to each factor $\SL_r(\Fp_p)$ is equal to the whole
product (this type of result is known as Goursat's Lemma, see,
e.g.,~\cite[Prop. 5.1]{chavdarov} or as Hall's Lemma~\cite[Lemma
3.7]{dunfield-thurston}). Again a similar property holds if the
Zariski closure of $\Gamma$ is an almost simple, connected,
simply-connected algebraic group.
\par
(3) In particular, the Independence Assumption holds with
$T_1(\Gamma)=\emptyset$ for the mapping class group of $\Sigma_g$
acting on the homology of the surface, because Goursat's Lemma applies
to the finite groups $\Sp_{2g}(\Fp_p)$.
\par
(4) The Independence Assumption may fail, for instance in the context
of orthogonal groups, when there is a global invariant which can be
read off any reduction. The simplest example of such an invariant is
the determinant: if $\Gamma\subset \GL_r(\Zz)$ is not contained in
$\SL_r(\Zz)$, the compatibility condition
$$
\det(\pi_p(g))=\det(g)\in \{\pm 1\}\subset \Fp_p^{\times}
$$
valid for all $p$ and $g\in \GL_r(\Zz)$ shows that the image of
$\pi_I$ is always contained in the proper subgroup
$$
\{(g_p)\in \Gamma_I\,\mid\, \det(g_p)=\det(g_q)\text{ for all $p$,
  $q\in I$}\}
$$
(identifying all copies of $\{\pm 1\}$). This issue appears,
concretely, in the example of the Apollonian group and Apollonian
circle packings, since the latter is a subgroup of an indefinite
orthogonal group intersecting both cosets of the special orthogonal
group, see~\cite{fuchs,f-s} for a precise analysis of this case,
and~\cite{fuchs-bams} for a survey.
\par
It should be emphasized that this failure of the Independence
Assumption is not dramatic: one can replace $\Gamma$ by $\Gamma\cap
\SL_r$ for instance, or by the other coset of the determinant (with
some adaptation since this is not a group).
\end{example}

We can now define the sifted sets $\sifted\subset \Gamma$ constructed by
inclusion-exclusion using local information: given a set $\mathcal{P}$
of primes (usually finite), and subsets
$$
\Omega_p\subset \Gamma_p
$$
for $p\in \mathcal{P}$, we let
$$
\sifted=\sifted(\mathcal{P};\Omega)= \{g\in\Gamma\,\mid\,
\pi_p(g)\notin \Omega_p\text{ for all } p\in\mathcal{P}\}=
\bigcap_{p\in\mathcal{P}}(\Gamma-\pi_p^{-1}(\Omega_p)).
$$
\par
We want to know something about the size, or maybe more ambitiously
the structure, of such sifted sets. In fact, quite often, we wish to
study sets which are not exactly of this shape, but are closely
related. 
\par
Frequently, we have an integer parameter $Q\geq 1$, and we take
$\mathcal{P}=\{p\leq Q\}$, the set of primes up to $Q$. In that case,
we will often denote $\sifted(Q;\Omega)=\sifted(\mathcal{P};\Omega)$,
and we may even sometimes simplify this to $\sifted(Q)$ if it is clear
that $\Omega$ is fixed.

\begin{example}\label{ex-first}
(1) Let $\Gamma=\Zz$, and let $\Omega_p=\{0,-2\}\subset \Fp_p$ for all
primes $p\leq Q$, where $Q\geq 2$ is some parameter. Taking
$\mathcal{P}=\{p\leq Q\}$, we have by
definition
$$
\sifted(Q)=\sifted(\mathcal{P};\Omega)=\{n\in\Zz\,\mid\, \text{neither
  $n$ nor $n+2$ has a prime factor } \leq Q\}.
$$
\par
In particular, for $N\geq 1$, the initial segment $\sifted(Q)\cap
\{1,\ldots,N\}$ contains all ``twin primes'' $n$ between $Q$ and $N$,
i.e., all primes $p$ with $Q<p\leq N$ such that $p+2$ is also
prime. Hence an upper-bound on the size of this initial segment will
be an upper-bound for the number of twin primes in this range. This is
valid independently of the value of $Q$. Furthermore, if $Q\geq
\sqrt{N+2}$, we have in fact equality: an integer $n\in
\sifted(\sqrt{N+2})\cap\{1,\ldots, N\}$ must be prime, as well as
$n+2$, since both integers only have prime factors larger than their
square-root. More generally, if $Q=N^{\beta}$ for some $\beta>0$, we
see that $\sifted(Q)\cap \{1,\ldots,N\}$ contains only integers $n$
such that both $n$ and $n+2$ have less than $1/\beta$ prime factors.
\par
(2) The first example is the prototypical example showing how sieve
methods are used to study prime patterns of various type. Bourgain,
Gamburd and Sarnak~\cite{bgs} extended this type of questions to
discrete subgroups of $\GL_r(\Zz)$. We present here a special case of
what is called the \emph{affine linear sieve} or the \emph{sieve in
  orbits}. There will be a few other examples below, and we refer to
the original paper or to~\cite{bourbaki} for a more general approach.
\par
We assume for simplicity, as before, that $\Gamma$ is Zariski-dense in
$\SL_r$. Let
$$
f\,:\, \SL_r(\Zz)\lra \Zz
$$
be a non-constant polynomial function, for instance the product of the
coordinates. We want to study the multiplicative properties of the
integers $f(g)$ when $g$ runs over $\Gamma$. Consider
\begin{equation}\label{eq-sifted-discrete}
\Omega_p=\{g\in \Gamma_p\,\mid\, f(g)\equiv 0\mods{p}\}\subset
\Gamma_p\subset \SL_r(\Fp_p),
\end{equation}
for $p\leq Q$. Then $\sifted(Q;\Omega)$ (recall that this is the
sifted set for $\mathcal{P}=\{p\leq Q\}$) is the set of $g\in \Gamma$
such that $f(g)$ has no prime factor $\leq Q$. In particular, for any
$\Delta>0$, the intersection
$$
\sifted(Q;\Omega)\cap \{g\in\Gamma\,\mid\, |f(g)|\leq Q^\Delta\}
$$
consists of elements where $f(g)$ has $<\Delta$ prime factors. For
instance, when $f$ is the product of coordinates, this set contains
elements $g\in \Gamma$ where all coordinates have less than $\Delta$
prime factors.
\par
(3) For our last example, consider the mapping class group $\Gamma$ of
$\Sigma_g$. Let $\mathcal{H}_g$ be a handlebody with boundary
$\Sigma_g$. For a mapping class $\phi\in \Gamma$, we denote by
$\mathcal{M}_{\phi}$ the compact $3$-manifold obtained by Heegaard splitting
using $\mathcal{H}_g$ and $\phi$, i.e., it is the union of two copies
of $\mathcal{H}_g$ where the boundaries are identified using (a
representative of) $\phi$ (see~\cite{dunfield-thurston} for more about
this construction).
\par
The image $J$ of $H_1(\mathcal{H}_g,\Zz)\simeq \Zz^g$ in
$H_1(\Sigma_g,\Zz)\simeq \Zz^{2g}$ is a lagrangian subspace (i.e., a
subgroup of rank $g$ such that the intersection pairing is identically
zero on $J$). We denote by $J_p\subset \Fp_p^{2g}$ its reduction
modulo $p$. It follows from algebraic topology that
\begin{gather*}
H_1(\mathcal{M}_{\phi},\Zz)\simeq H_1(\Sigma_g,\Zz)/\langle J,\phi\cdot
J\rangle,\\
H_1(\mathcal{M}_{\phi},\Fp_p)\simeq H_1(\mathcal{M}_{\phi},\Zz)\otimes
\Fp_p\simeq H_1(\Sigma_g,\Fp_p)/\langle J_p,\phi\cdot J_p\rangle.
\end{gather*}
\par
Thus if we let
\begin{equation}\label{eq-omega-h1}
\Omega_p=\{\gamma\in \Sp_{2g}(\Fp_p)\,\mid\, \gamma \cdot J_p\cap
J_p=\emptyset\}= \{\gamma\in \Sp_{2g}(\Fp_p)\,\mid\, \langle
J_p,\gamma \cdot J_p\rangle=\Fp_p^{2g}\},
\end{equation}
we see that any sifted set $\sifted(\mathcal{P}; \Omega)$ contains all
mapping classes such that $\mathcal{M}_{\phi}$ has first rational Betti number
positive. 
\par
We will discuss this example further in Section~\ref{sec-large}. The
reader who is not familiar with sieve is however encouraged to try to
find the answer to the following question: What is the great
difference that exists between this example and the previous ones?
\end{example}

\section{Conditions for sieving}\label{sec-implement}

Having defined sifted sets and seen that they contain information of
great potential interest, we want to say something about them. The
basic question is ``How large is a sifted set $\sifted$?'' In order to
make this precise, some truncation of $\sifted$ is needed, since in
general this is (or is expected to be) an infinite set. In fact, we
saw in the simplest examples (e.g., twin primes) that this truncation
(in that case, the consideration of an initial segment of a sifted
set) is crucially linked to deriving interesting information from
$\sifted$, as one needs usually to handle a truncation which is
correlated with the size of the primes in the set $\mathcal{P}$
defining the sieve conditions.
\par
When sieving in the generality we consider, it is a striking fact that
there are different ways to truncate the sifted sets, or indeed to
measure subsets of $\Gamma$ in general (although those we describe
below seem, ultimately, to be closely related.) We will speak of
``counting methods'' below to refer to these various truncation
techniques.
\par
\medskip
\par
\underline{Method 1.} [Archimedean balls] Fix a norm $\|\cdot \|$ (or
some other metric) on the ambient Lie group $\GL_r(\Rr)$ (for instance
the operator norm as linear maps on euclidean space, but other choices
are possible) and consider
$$
\sifted\cap \{g\in \Gamma\,\mid\, \|g\|\leq T\}
$$
for some parameter $T\geq 1$. This is a finite set, and one can try to
estimate (from above or below, or both) its cardinality.

\begin{example}\label{ex-sieve-orbits-1}
  Let $\Gamma$ be a Zariski-dense subgroup of $\SL_r(\Zz)$ and $f$ a
  non-constant polynomial function on $\SL_r(\Zz)$. For some $d\geq
  1$, we have
$$
|f(g)|\ll \|g\|^d
$$
for all $g\in\Gamma$. Hence if we consider the sifted
set~(\ref{eq-sifted-discrete}) for $Q=T^{\beta}$, the elements in 
$$
\sifted(Q)\cap \{g\in \Gamma\,\mid\, \|g\|\leq T\}
$$
are such that $f(g)$ has at most $d/\beta$ prime factors.
\end{example}

Counting in archimedean balls in subgroups of arithmetic groups, even
without involving sieve, is a delicate matter, especially in the
sparse case, which involves deep ideas from spectral theory, harmonic
analysis and ergodic theory. We refer to the book of Gorodnik and
Nevo~\cite{gn} for the case of arithmetic groups, and to Oh's
surveys~\cite{oh} and~\cite{oh-2} for the sparse case, as well as to
the recent paper of Mohamadi and Oh~\cite{mohamadi-oh} concerning
geometrically finite subgroups of isometries of hyperbolic spaces.
\par
\medskip
\par
\underline{Method 2.} [Combinatorial balls] Since the groups $\Gamma$
of interest are most often finitely generated, and indeed sometimes
given with a set of generators, one can replace the archimedean metric
of the first method with a combinatorial one. Thus if $S=S^{-1}$ is a
generating set of $\Gamma$, we denote by $ \ell_S(g)$ the word-length
metric on $\Gamma$ defined using $S$. The sets
$$
\sifted\cap \{g\in \Gamma\,\mid\, \ell_S(g)\leq T\},\quad\text{ or }
\quad
\sifted\cap \{g\in \Gamma\,\mid\, \ell_S(g)=T\},
$$
are again finite, and one can attempt to estimate their size. 
\par
This method is particularly interesting when $S$ is a set of free
generators of $\Gamma$ (and their inverses), because one knows
precisely the size of the balls for the combinatorial metric in that
case. And even if this is not the case, one can often find a subgroup
of $\Gamma$ which is free of rank $\geq 2$, and use this subgroup
instead of the original $\Gamma$ (this technique is used
in~\cite{bgs}; in that case, the necessary free subgroup is found
using the Tits Alternative, a very specific case of which says that if
$\Gamma$ is Zariski-dense in $\SL_r$, then it contains a free subgroup
of rank $2$.)
\par
\medskip
\par
\underline{Method 3.} [Random walks] Instead of trying to reduce to
free groups using a subgroup, one can replace $\Gamma$ by the free
group $F(S)$ generated by $S$ and use the obvious homomorphisms
$$
\phi\,:\, F(S)\lra \Gamma\lra \GL_r(\Zz)
$$
and
$$
F(S)\lra \Gamma\lra \Gamma_p
$$
to define sieve problems and sifted sets. An alternative to this
description is to use the generating set $S$ and count elements in
balls for the word-length metric $\ell_S$ \emph{with multiplicity},
the multiplicity being the number of representations of $g\in \Gamma$
by a word of given (or bounded) length. This means one measures the
size of a set $X\subset \Gamma$ truncated to the sphere of radius
$N\geq 1$ around the origin by its density
$$
\mu_N(X)=\frac{1}{|S|^N}|\{(s_1,\ldots,s_N)\in S^N\,\mid\, s_1\cdots
s_N\in X\}|
$$
and therefore one tries to measure the density of the sifted set
$\mu_N(\sifted)$, as a way of measuring its size within a given
ball. If one wishes to measure balls instead of spheres, a simple
expedient is to replace $S$ by $S_1=S\cup \{1\}$ (since the sphere of
radius $N$ for $\ell_{S_1}$ is the ball of radius $N$ for $\ell_S$).
\par
It is often convenient to think of this in terms of a random walk: one
assumes given, on a probability space $\Omega$, a sequence of
independent $S$-valued random variables $\xi_n$, and one defines a
random walk $(\gamma_n)$ on $\Gamma$ by
$$
\gamma_0=1,\quad\quad \gamma_{n+1}=\gamma_n\xi_{n+1}\text{ for }
n\geq 0.
$$
\par
If all steps $\xi_n$ are uniformly distributed on $S$, it follows that
$$
\mu_N(X)=\proba(\gamma_N\in X),
$$
or in other words, the density $\mu_N$ is the probability distribution
of the $N$-th step of this random walk. 

\begin{example}\label{ex-sieve-orbits-2}
  The analogue (for Methods 2 and 3) of the argument in
  Example~\ref{ex-sieve-orbits-1} is the following: given a function
  $f$ as in that example, there exists $C\geq 1$ such that, for all
  $g\in\Gamma$, we have
$$
|f(g)|\leq C^{\ell_S(g)}
$$
(simply because the operator norm of $g$ is submultiplicative and
hence grows at most exponentially with the word-length metric). Thus
elements which have word-length at most $N$ and belong to a sifted set
$\sifted(Q;\Omega)$ with $Q$ of size $A^N$, for some $A>1$, have at
most $(\log A)/(\log C)$ prime factors.
\end{example}

\begin{example}[Dunfield--Thurston random manifolds]\label{ex-dt-1}
This third counting method is the least familiar to classical analytic
number theory. This random walk approach was however already
considered by Dunfield--Thurston~\cite{dunfield-thurston} as a way of
studying random $3$-manifolds, using the Heegard-splitting
construction based on mapping class groups as in
Example~\ref{ex-first}, (3): given an integer $g\geq 1$, they
consider a finite generating set $S$ of the mapping class group
$\Gamma$ of $\Sigma_g$ and the associated random walk $(\phi_n)$. The
$3$-manifolds $\mathcal{M}_{\phi_n}$ are then ``random $3$-manifolds'' and some
of their properties can be studied using sieve methods. 
\end{example}
\par
\medskip
\par
It is of course useful to have a way of considering these three
methods in parallel. This can be done by assuming that one has a
sequence $(\mu_N)$ of finite measures on $\Gamma$, and by considering
the problem of estimating $\mu_N(\sifted)$, the measure of the sifted
set. In Method 1, these measures would be the uniform counting measure
on the intersection of $\Gamma$ with the balls of radius $N$ in
$\GL_r(\Rr)$, in Method 2, the uniform counting measure on the
combinatorial ball of radius $N$, and in Method 3, the probability law
of the $N$-th step of the random walk.

\section{Implementing sieve with expanders}

We will now explain how all this relates to expanders. The one-line
summary is that the expander condition will allow us to apply
classical results of sieve theory to settings of discrete groups
``with exponential growth'' (one might prefer to say, ``in
non-amenable settings''). We can motivate this convincingly as
follows.
\par
The simplest possible sieve problem occurs when the set $\mathcal{P}$
of conditions is restricted to a single prime, and one is asking for
$$
\mu_N(\{g\in\Gamma\,\mid\, \pi_p(g)=g_0\})
$$
for a fixed prime $p$ and a fixed $g_0\in \Gamma_p$. One sees that,
assuming $p$ is fixed, this elementary-looking question concerns the
distribution of the image of the sequence $\pi_{p,*}\mu_N$ of measures
on the finite group $\Gamma_p$. This may well be expected to have a
good answer.

\begin{example}
Consider (one last time) the classical case $\Gamma=\Zz$. If we
truncate by considering initial segments $\{1,\ldots, N\}$, we are
asking here about the number of positive integers $\leq N$ congruent
to a given $a$ modulo $p$. The proportion of these converges of course
to $1/p$, and this is usually so self-evident that one never mentions
it specifically. (But, still in classical cases, note that if one
starts the sieve from the set of primes instead of $\Zz$, then this
basic question is resolved by Dirichlet's Theorem on primes in
arithmetic progressions, and the uniformity in this question is
basically the issue of the Generalized Riemann Hypothesis.)
\end{example}

On intuitive grounds as well as theoretically, one can expect that the
``probability'' that $g$ reduces modulo $p$ to $g_0$ should be about
$1/|\Gamma_p|$. This amounts to expecting that the probability
measures $\pi_{p,*}(\mu)/\mu_N(\Gamma)$ converge weakly to the uniform
(Haar) probability measure on this finite group. It is when
considering uniformity of such convergence that expander graphs enter
the picture.
\par
We can already deduce from this intuition the following heuristic
concerning the size of a sifted set $\sifted(\mathcal{P};\Omega)$:
each condition $\pi_p(g)\notin \Omega_p$ should hold with
``probability'' approximately 
$$
1-\frac{|\Omega_p|}{|\Gamma_p|},
$$
and these sieving conditions, for distinct primes, should be
independent. Hence one may expect that
\begin{equation}\label{eq-heuristic}
\mu_N(\sifted(\mathcal{P};\Omega))\approx
\mu_N(\Gamma)\prod_{p\in\mathcal{P}}{
\Bigl(1-\frac{|\Omega_p|}{|\Gamma_p|}\Bigr)
}
\end{equation}
(where the symbol $\approx$ here only means that the right-hand side
is a first guess for the left-hand side...)
\par 
The simplest counting method to explain this is Method 3, where the
argument is very transparent. We therefore assume in the remainder of
this section that $\mu_N$ is the probability law of the $N$-th step of
a random walk on $\Gamma$ as above.
\par
It is then an immediate corollary of the theory of finite Markov
chains (applied to the random walk on the Cayley graph of $\Gamma_p$
induced by that on $\Gamma$) that, if $1\in S$ (or more generally if
this Cayley graph is not bipartite, i.e., if there exists no
surjective homomorphism $\Gamma_p\lra \{\pm 1\}$ such that each
generator $s\in S$ maps to $-1$), we have exponentially-fast
convergence to the probability Haar measure. Precisely, let $M_p$ be
the Markov operator acting on functions on $\Gamma_p$ by
$$
(M_p\varphi)(x)=\frac{1}{|S|}\sum_{s\in S}{\varphi(xs)}.
$$
\par
This operator also acts on functions of mean $0$, i.e., on the space
$L^2_0(\Gamma_p)$ of functions such that
$$
\sum_{g\in\Gamma_p}{\varphi(g)}=0,
$$
and has real eigenvalues. Let $\rho_p<1$ be its spectral radius (it is
$<1$ because the eigenvalue $1$ is removed by restricting to $L^2_0$,
while $-1$ is not an eigenvalue because the graph is not
bipartite). We then have
$$
\Bigl|\mu_N(\pi_p(g)=g_0)-
\frac{1}{|\Gamma_p|}
\Bigr|\leq \rho_p^N
$$
for all $N\geq 1$.
\par
More generally, under the Independence Assumption~\ref{hyp-indep}, if
$I$ is a finite set of primes not in $T_1(\Gamma)$, the same argument
applied to the quotient
$$
\Gamma\lra \Gamma_I=\prod_{p\in I}{\Gamma_p}
$$
shows that that for any $(g_p)\in \Gamma_I$, we have
\begin{equation}\label{eq-first-equidistribution}
\Bigl|\mu_N\Bigl(\pi_p(g)=g_p\text{ for } p\in I\Bigr)- \prod_{p\in
  I}\frac{1}{|\Gamma_p|} \Bigr|\leq \rho_I^N
\end{equation}
where $\rho_I<1$ is the corresponding spectral radius for
$\Gamma_I$. It follows by summing over $x=(g_p)\in\Gamma_I$ that we
have a quantitative equidistribution
\begin{equation}\label{eq-quant-equid}
  \int{\varphi((\pi_p(g))_{p\in I})d\mu_N(g)}=
  \frac{1}{|\Gamma_I|}
  \sum_{x\in \Gamma_I}\varphi(x)+O(|\Gamma_I|\|\varphi\|_{\infty}\rho_I^N)
\end{equation}
(with an absolute implied constant) for any function $\varphi$ on
$\Gamma_I$. 
\par
In particular, we see that if $\mathcal{P}$ is a fixed set of primes
(not in $T_1(\Gamma)$), then as $N\ra +\infty$, the basic
heuristic~(\ref{eq-heuristic}) is valid asymptotically:
\begin{equation}\label{eq-bounded-sieve}
\lim_{N\ra +\infty}
\mu_N(\sifted(\mathcal{P};\Omega))=
\lim_{N\ra +\infty}
\proba(\gamma_N\in \sifted(\mathcal{P};\Omega))=
\prod_{p\in\mathcal{P}}
{
\Bigl(1-\frac{|\Omega_p|}{|\Gamma_p|}\Bigr)
}
\end{equation}
(we will call this a ``bounded sieve'' statement).
\par
The difficulty (and fun!) of sieve methods is that the sifted sets of
most interest are such that the primes involved in $\mathcal{P}$ are
\emph{not} fixed as $N\ra +\infty$: they are in ranges increasing with
the size of the elements being considered (as shown already by the
example of the twin primes). It is clear that in order to handle such
sifted sets, we need a uniform control of the equidistribution
properties modulo primes, and modulo finite sets of primes
simultaneously.  The best we can hope for is
that~(\ref{eq-first-equidistribution}) hold with the spectral radius
bounded away from one \emph{independently of $I$}. This is, of course,
exactly the conditions under which the family of Cayley graphs of
$\Gamma_I$ with respect to the generators $S$ is a family of
(absolute) expander graphs.

\begin{remark}
  We have discussed the example of the random walk counting method. It
  is a fact that analogues of~(\ref{eq-first-equidistribution}) hold
  in all cases where sieve methods have been successfully
  applied. Moreover, these analogues hold uniformly with respect to
  $I$, and ultimately, the source is always equivalent to the
  expansion property of the Cayley graphs, although the proofs and the
  equivalence might be much more involved than the transparent
  argument that exists in the case of random walks.
\end{remark}

\begin{example}
  The first case beyond the classical examples (or the case of
  arithmetic groups, where Property $(T)$ or $(\tau)$ can be
  used,\footnote{\ See the works of Gorodnik and Nevo~\cite{gn2} for
    the best known in this direction.}  although this also had not
  been done before) where sieve in discrete groups was implemented is
  due to Bourgain, Gamburd and Sarnak~\cite{bgs}, who (based on
  earlier work of Helfgott~\cite{helfgott} and
  Bourgain--Gamburd~\cite{bg}) proved that if $\Gamma$ is a
  finitely-generated Zariski-dense subgroup of $\SL_2(\Zz)$ (or even
  of $\SL_2(\Oc)$, where $\Oc$ is the ring of integers in a number
  field), the Cayley graphs of $\Gamma_I$, where $I$ runs over finite
  subsets of $T_1(\Gamma)$, form a family of (absolute) expanders. The
  problem of generalizing this to $\SL_r$, or to Zariski-dense
  subgroups of other algebraic groups, was one of the motivations for
  the recent developments of this result, and of the basic ``growth''
  theorem of Helfgott, to more general groups. We now know an
  essentially best possible result (see~\cite{sg-v, varju}, and the
  surveys~\cite{salehi-golsefidy, breuillard, pyber} of
  Salehi-Golsefidy, Breuillard and Pyber--Szab\'o in these Proceedings
  for introductions to this area):

\begin{theorem}[Salehi-Golsefidy--Varj\'u]\label{th-expansion}
  Let $\Gamma\subset \GL_r(\Zz)$ be finitely generated by $S=S^{-1}$,
  with Zariski-closure $\Gg$. For $p$ prime, let $\Gamma_p$ be the
  image of $\Gamma$ under reduction modulo $p$, and for a finite set
  of primes $I$, let $\Gamma_I$ be the image of $\Gamma$ in
$$
\prod_{p\in I}\Gamma_p,
$$
under the simultaneous reduction homomorphism.
\par
If the connected component of the identity in $\Gg$ is a perfect
group, then there exists a finite set of primes $T_1(\Gamma)$ such
that the family of Cayley graphs of $\Gamma_I$, for $I\cap
T_1(\Gamma)=\emptyset$, is an expander family. 
\end{theorem}

\end{example}

We can now describe what is the implication of some classical sieve
results in the context of discrete groups.  We assume formally the
following:

\begin{assumption}[Expansion]\label{hyp-exp}
  There exists a finite set of primes $T_2(\Gamma)$ such that $\Gamma$
  satisfies the Independence Assumption~\ref{hyp-exp} for primes not
  in $T_2(\Gamma)$, and furthermore the family of Cayley graphs of
  $\Gamma_I$, for $I\cap T_2(\Gamma)=\emptyset$, is an expander
  family, i.e., there exists $\rho<1$, such that for any finite set
  $I$ of primes $p\notin T_2(\Gamma)$, the spectral radius for the
  Markov operator on $\Gamma_I$ satisfies
$$
\rho_I\leq \rho.
$$
\end{assumption}

By~(\ref{eq-first-equidistribution}), this assumption implies that the
asymptotic formula
\begin{equation}\label{eq-uniform-equid}
\proba(\pi_p(\gamma_n)=g_p\text{ for } p\in I\Bigr)\sim \prod_{p\in
  I}\frac{1}{|\Gamma_p|}
\end{equation}
holds uniformly for $n\geq 1$ and sets $I$ such that $|\Gamma_I|\leq
\tilde{\rho}^{-n}$, for any $\tilde{\rho}>\rho$. If we assume that
\begin{equation}\label{eq-not-too-large}
|\Gamma_p|\leq p^B
\end{equation}
for some fixed $B\geq 1$, this means that we can control
simultaneously and uniformly all reductions of the $N$-th step as long
as $q_I\leq \tilde{\rho}^{-n/B}$. Note that~(\ref{eq-not-too-large})
is not very restrictive: it holds (with $B=r^2$) if $\pi_p$ is just
the reduction modulo $p$ on $\GL_r(\Zz)$, which is the case in all our
applications.
\par
The most classical types of sieve are those when the sieving
conditions determined by $\Omega_p$ hold with probability
approximately $\kappa/p$, at least on average, were $\kappa$ is a
fixed real number traditionally called the \emph{dimension} of the
sieve. Precisely, we say that $(\Omega_p)$ is of dimension $\kappa$ if
we have\footnote{\ There are other weaker conditions that are enough  to allow an efficient sieve, but we refer only to~\cite[\S 5.5]{FI}
  for a discussion of these aspects.}
\begin{equation}\label{eq-dimension}
  \sum_{p\leq X}\frac{|\Omega_p|}{|\Gamma_p|}\log p
  =\kappa \log X + O(1)
\end{equation}
for $X\geq 2$. Note that this is certainly true, by the Prime Number
Theorem, if
$$
\frac{|\Omega_p|}{|\Gamma_p|}=\frac{\kappa}{p}+
O\Bigl(\frac{1}{p^{1+\delta}}\Bigr) 
$$
for some $\delta>0$ and all $p$ prime.
\par
We then have the following basic result:

\begin{theorem}[Small sieve in discrete groups]\label{th-small}
  Let $\Gamma$ be a discrete group finitely generated by $S=S^{-1}$,
  given with $\phi\,:\, \Gamma\lra \GL_r(\Zz)$ and surjective
  homomorphisms $\pi_p$ to finite groups $\Gamma_p$ as above, in
  particular with~\emph{(\ref{eq-not-too-large})} for some fixed
  $B\geq 1$. Assume that $\Gamma$ satisfies the Independence
  Assumption~\ref{hyp-indep} and the Expansion
  Assumption~\ref{hyp-exp}. Let $(\gamma_n)$ denote a random walk on
  $\Gamma$ using steps from $S$, and let $\mu_n$ denote the
  probability law of the $n$-th step.  Let $\Omega_p\subset \Gamma_p$
  be finite sets such that~\emph{(\ref{eq-dimension})} holds for some
  $\kappa>0$.
\par
There exists $A>0$ such that, for all $n\geq 1$, if we let $Q=A^n$ and
take $\mathcal{P}$ to be the set of primes $p\leq Q$ with $p\notin
T_2(\Gamma)$, then we have
$$
\proba(\gamma_n\in \sifted(\mathcal{P};\Omega)) \asymp
\frac{1}{n^{\kappa}}
$$
for all $N$ large enough.
\end{theorem}

This is essentially a direct consequence of the standard Brun-type
sieve, building on the Independence and Expansion assumptions. The
mechanism is explained in~\cite{bourbaki}, and to avoid repetition, we
will not give further details here. We simply add a few
remarks. First, this result confirms the
heuristic~(\ref{eq-heuristic}) as far as the order of magnitude is
concerned, i.e., up to multiplicative constants. Indeed, the
right-hand side of~(\ref{eq-heuristic}) is, in this case, given by
$$
\prod_{\stacksum{p\leq A^n}{p\notin T_2(\Gamma)}}
{\Bigl(1-\frac{|\Omega_p|}{|\Gamma_p|}\Bigr)
}
\sim \prod_{p\leq A^n}
{\Bigl(1-\frac{\kappa}{p}\Bigr)
}
\asymp n^{-\kappa}
$$
as $n\ra +\infty$, by~(\ref{eq-dimension}) and the Mertens Formula (or the
Prime Number Theorem.) Secondly, the result is best possible in the
sense that one cannot replace the inequalities up to multiplicative
constants by an asymptotic formula in this generality (this is also
seen from the Mertens Formula and the Prime Number Theorem). Finally,
the result is by no means an easy consequence
of~(\ref{eq-first-equidistribution}) and the uniformity afforded by
expansion. 

\begin{example}[Sieve in orbits]
We illustrate the above result by deriving, as a corollary, a special
case of the sieve in orbits (or affine linear sieve) of~\cite{bgs}.
\par
Let $\Gamma$ be Zariski-dense in $\SL_r(\Zz)$ with $r\geq 2$, and
generated by the finite set $S=S^{-1}$. We take for $\pi_p$ the
reduction maps. Let
$$
f\,:\, \SL_r(\Zz)\lra \Zz
$$
be a non-constant polynomial map and let $\Omega_p\subset
\SL_r(\Fp_p)$ be the set of zeros of $f$. Since $f$ is non-constant,
the algebraic subvariety $Z_f$ of $\SL_r$ defined by the equation
$f=0$ is a hypersurface in $\SL_r$. Relatively elementary
considerations of algebraic geometry, together with the Lang-Weil
estimates for the number of points on algebraic varieties over finite
fields, show that we have
\begin{equation}\label{eq-kappa-orbits}
\frac{|\Omega_p|}{|\Gamma_p|}=\frac{\kappa_p}{p}+O(p^{-3/2})
\end{equation}
for some $\kappa_p\geq 0$ which depends on the splitting of $p$ in the
field of definition of the geometrically irreducible components of
$Z_f$ (if all geometrically irreducible components of $Z_f$ are
defined over $\Qq$, then $\kappa_p$ is the number of these irreducible
components, as is well-known; the general case is carefully explained
by Salehi-Golsefidy and Sarnak~\cite[Prop. 15, Cor. 17]{sg-s}).  A
further application of the Chebotarev density theorem (see~\cite[Lemma
21]{sg-s}) shows that
$$
\sum_{p\leq X}{\kappa_p}=\kappa \pi(X)+O(X/(\log X)^2)
$$
where $\kappa$ is the number of $\Qq$-irreducible components of $Z_f$
(if all geometrically irreducible components are defined over $\Qq$,
we have $\kappa_p=\kappa$ for all but finitely many $p$).

\begin{example}
(1) Consider the function
$$
f\Bigl(\begin{pmatrix}a & b\\c &d
\end{pmatrix}
\Bigr)=a^2+d^2
$$
on $\SL_2$. For $p\equiv 3\mods{4}$, we have $\kappa_p=0$ (since
$a^2+d^2=0\in\Fp_p$ implies $a=d=0$ in this case), while for $p\equiv
1\mods{4}$, we have $\kappa_p=2$, reflecting the fact that $Z_f$ has
then, over $\Fp_p$, the two geometrically irreducible components
defined by $a+\eps d=0$ and $a-\eps d=0$, where $\eps^2=-1$. The
average of $\kappa_p$ over $p$ is then equal to $\kappa=1$.
\par
(2) Consider the function
$$
f(g_{i,j})=\prod_{i,j} g_{i,j}
$$
on $\SL_r$. Then the irreducible components are defined by $g_{i,j}=0$
for a fixed $(i,j)$, and are all defined over $\Qq$. Thus we have
$\kappa=\kappa_p=r^2$ for all primes $p$.
\end{example}

Thus all assumptions of Theorem~\ref{th-small} hold (the Expansion
Assumption coming from Theorem~\ref{th-expansion}), and we deduce that
there exists a finite set of primes $T$ and $A>1$ such that if
$\mathcal{P}$ is the set of primes not in $T$ and $\leq A^n$, we have
$$
\proba(\gamma_n\in \sifted(\mathcal{P};\Omega))\asymp n^{-\kappa}.
$$
\par
Using Example~\ref{ex-sieve-orbits-2}, we therefore deduce:

\begin{theorem}[Sieve in orbits; Bourgain--Gamburd--Sarnak]\label{th-bgs}
  Let $\Gamma$ and $f$ be as above. There exists $\omega\geq 1$ such
  that the set $\Oc_{f}(\omega)$ of all $g\in\Gamma$ such that $f(g)$
  has at most $\omega$ prime factors satisfies
\begin{equation}\label{eq-affine-sieve}
\proba(\gamma_n\in \Oc_f(\omega))\asymp n^{-\kappa}
\end{equation}
for $n$ large enough.
\end{theorem}

One of the insights of Bourgain--Gamburd--Sarnak was that such a
statement has a more qualitative corollary which is already very
interesting and doesn't require any consideration of a special
counting method:

\begin{corollary}\label{cor-zariski-dense}
  Let $\Gamma$ and $f$ be as above. There exists $\omega\geq 1$ such
  that the set $\Oc_{f}(\omega)$ is Zariski-dense in $\SL_r$.
\end{corollary}

\begin{proof}
  It is enough to check that if a subset $X\subset \Gamma$ is not
  Zariski-dense, then a lower-bound
$$
\proba(\gamma_n\in X)\gg n^{-\kappa}
$$
does not hold for any $\kappa>0$, since $\Oc_f(\omega)\subset Z$ would
then contradict the sieve lower bound~(\ref{eq-affine-sieve}) (note
that here $(\gamma_n)$ is just an auxiliary tool).
\par
Given $X$, there exists a non-trivial polynomial $f$ such that
$X\subset Z_f$ (recall that this is the zero set of $f$). Then, for
any prime $p$ (large enough so that reduction of $f$ modulo $p$ makes
sense) the image of $X$ modulo $p$ is contained in the zeros of $f$
modulo $p$. But using~(\ref{eq-kappa-orbits}) and
summing~(\ref{eq-uniform-equid}) over the zeros of $f$ modulo $p$, we
have
$$
\proba(\pi_p(\gamma_n)\in Z_f\mods{p})\sim \kappa_p p^{-1}
$$
uniformly for $p\leq A^n$ for some $A>1$. Taking $p$ of size $A^n$,
we deduce
$$
\proba(\gamma_n\in X)\leq \proba(\pi_p(\gamma_n)\in Z_f\mods{p})
\ll A^{-n}
$$
for $n$ large enough. Thus the probability to be a zero of a given
function $f$ is in fact exponentially small for a long walk, and this
contradicts the lower bounds for $\Oc_f(\omega)$.
\end{proof}

In fact, as noted in~\cite{bourbaki} and as we will see in the next
section, this has a natural refinement where Zariski-dense is replaced
by ``not thin'' in the sense of Serre.
\par
Note that Salehi-Golsefidy and Sarnak~\cite{sg-s} have extended the
basic small sieve statement to much more general groups, not
necessarily reductive, using the full power of
Theorem~\ref{th-expansion} together with special considerations to
handle unipotent groups.
\end{example}

\begin{example}\label{ex-dt-2}
  Theorem~\ref{th-small} also applies in the context of
  Dunfield--Thurston manifolds, as in Example~\ref{ex-dt-1}. Indeed,
  the Expansion Assumption~\ref{hyp-exp} is here a consequence of
  Property $(T)$ for $\Sp_{2g}(\Zz)$. As observed in~\cite{bourbaki},
  a consequence of Theorem~\ref{th-small} which is similar in spirit
  to the affine linear sieve is that there exists $\omega\geq 1$ such
  that
$$
\proba(H_1(\mathcal{M}_{\phi_n},\Zz)\text{ is finite and has order
  divisible by $\leq \omega$ primes})\asymp n^{-1}
$$
for $n$ large enough. (We recall that the genus $g$ defining the
Heegard splitting is fixed).
\end{example}

One can certainly use the sieve setting for many other purposes. As
one further example, we show how the method of Granville and
Soundararajan~\cite[Prop. 3]{granville-sound} gives a version of the
Erd\"os-Kac Theorem for discrete groups. For simplicity, we only state
the result for the affine sieve, and give one further example
afterwards (a version of this for curvatures of Apollonian circle
packings was proved by Djankovi\'c~\cite{djankovic}).

\begin{theorem}[Erdös-Kac central limit theorem for affine
  sieve]\label{th-erdos-kac}
  Assume that $\Gamma\subset \SL_r(\Zz)$ is Zariski-dense in $\SL_r$
  and $f$ is a non-constant polynomial function satisfying the
  assumptions of Theorem~\ref{th-bgs}. For a random walk $(\gamma_n)$
  on $\Gamma$, let $\omega_f(\gamma_n)=\omega(f(\gamma_n))$ if
  $f(\gamma_n)\not=0$, and $\omega_f(\gamma_n)=0$ otherwise. Then the
  random variables
$$
\frac{\omega_f(\gamma_n)-\kappa \log n}{\sqrt{\kappa \log n}}
$$
converge in law to the standard normal random variable as $n\ra
+\infty$.
% , where
% \begin{align*}
%   m_n(\Omega)&=\sum_{p\leq e^n}\frac{|\Omega_p|}{|\Gamma_p|}\sim\kappa
%   \log n,\\
%   \sigma_n(\Omega)^2&=\sum_{p\leq e^n}\frac{|\Omega_p|}{|\Gamma_p|}
%   \Bigl(1-\frac{|\Omega_p|}{|\Gamma_p|}\Bigr)\sim \kappa \log n.
% \end{align*}
\end{theorem}

\begin{proof}
  We proceed exactly as in~\cite{granville-sound}, leaving some
  details to the reader.  This uses the method of moments to prove
  convergence in law to the normal distribution: classical probability
  results imply that it is enough to prove that for all integers
  $k\geq 0$, we have
$$
\expect\Bigl(\Bigl(
\frac{\omega_f(\gamma_n)-\kappa \log n}{\sqrt{\kappa \log
    n}}
\Bigr)^k\Bigr)\lra c_k
$$
as $n\ra +\infty$, where $c_k=\expect(\mathcal{N}(0,1)^k)$ is the
$k$-th moment of a standard normal random variable. 
\par
We first deal with the possibility that $f(\gamma_n)\not=0$. By
bounding
$$
\proba(f(\gamma_n)=0)\leq \proba(f(\pi_p(\gamma))=0)
$$
for any prime $p$ large enough, and arguing as in the
proof of Corollary~\ref{cor-zariski-dense}, we get
$$
\proba(f(\gamma_n)=0)\ll c^{-n}
$$
for some $c>1$. Thus the expectation above, restricted to the set
$f(\gamma_n)=0$, is 
$$
\ll (\kappa\log n)^{k/2} c^{-n}\lra 0
$$
as $n\ra +\infty$. 
\par
Below we use the notation $\tilde{\expect}$ to denote expectation
restricted to $f(\gamma_n)\not=0$. We fix some integer $k\geq 0$, and
fix some auxiliary $A>1$. We will compare
$$
M_k=\tilde{\expect}( (\omega_f(\gamma_n)-\kappa \log n)^k)
$$
with the moment of ``truncated'' count of primes dividing
$f(\gamma_n)$ defined by
$$
N_k=\tilde{\expect}\Bigl(\Bigl(
\sum_{\stacksum{p\leq A^n}{\pi_p(\gamma_n)\in\Omega_p}}{1}
-\sum_{p\leq A^n}\delta_p\Bigr)^k\Bigr),
$$
where
$$
\delta_p=\frac{|\Omega_p|}{|\Gamma_p|},%%proba(\delta_p(\gamma_n)\in\Omega_p),
$$
and then estimate asymptotically this second moment when $A>1$ is
small enough with respect to $k$.
\par
For the first step, we note that when $f(\gamma_n)\not=0$, we have
$$
\omega_f(\gamma_n)-\kappa \log n=A_1+A_2+A_3
$$
where
\begin{gather*}
A_1=\sum_{\stacksum{p\leq A^n}{\pi_p(\gamma_n)\in\Omega_p}}{1}
-\sum_{p\leq A^n}\delta_p\\
A_2=\omega(f(\gamma_n))-\sum_{\stacksum{p\leq
    A^n}{\pi_p(\gamma_n)\in\Omega_p}}{1}\\
A_3=\sum_{p\leq A^n}{\delta_p}-\kappa\log n
\end{gather*}
\par
If $C>1$ is such that $|f(g)|\leq C^{\ell_S(g)}$, then we get
$$
0\leq A_2\leq \frac{\log C}{\log A},
$$
while, by~(\ref{eq-dimension}), we have
$$
A_3=\sum_{p\leq A^n}\delta_p-\kappa\log n=\sum_{p\leq A^n}
\frac{|\Omega_p|}{|\Gamma_p|}-\kappa\log n=O(1),
$$
so that $A_2+A_3$ is uniformly bounded for a fixed choice of
$A$. Using the multinomial theorem, it follows that
$$
M_k=N_k+O(\max_{j\leq k-1}\tilde{N}_j),
$$
where
$$
\tilde{N}_j=\tilde{\expect}\Bigl(\Bigl|
\sum_{\stacksum{p\leq A^n}{\pi_p(\gamma_n)\in\Omega_p}}{1}
-\sum_{p\leq A^n}\delta_p\Bigr|^k\Bigr).
$$
\par
We have $\tilde{N}_j=N_j$ is even and if $j$ is odd, we get
$$
\tilde{N}_j\leq \sqrt{\tilde{N}_{j-1}\tilde{N}_{j+1}}
=\sqrt{N_{j-1}N_{j+1}}
$$
by the Cauchy-Schwarz inequality, showing that good understanding of
$N_j$ for $j\leq k$ will suffice to estimate $M_k$.
\par
For the second step, we write
$$
X_p=\charfun_{\pi_p(\cdot)\in\Omega_p}-\delta_p
$$
for $p\leq A^n$, sum over $p$, and open the $k$-th power defining
$N_k$. Note that $|X_p|\leq 1$. Exchanging the multiple sum over
primes and the expectation, we get
$$
N_k=\multsum_{p_1,\ldots p_k\leq A^n}
\tilde{\expect}\Bigl(
\prod_{j=1}^k{X_{p_j}}
\Bigr).
$$
\par
For any fixed $(p_1,\ldots, p_k)$, we note that
$$
\tilde{\expect}\Bigl(
\prod_{j=1}^k{X_{p_j}}
\Bigr)=\expect\Bigl(
\prod_{j=1}^k{X_{p_j}}
\Bigr)-\expect\Bigl(
\prod_{j=1}^k{X_{p_j}}\charfun_{f(\gamma_n)=0}\Bigr)
$$
and the second term is bounded by $\proba(f(\gamma_n)=0)\ll c^{-n}$
since $0\leq X_{p_j}\leq 1$. Thus the total change in replacing
$\tilde{\expect}$ by $\expect$ in the formula above for $N_k$ is $\ll
A^{nk}c^{-n}$, which is negligible if $A$ is chosen small enough.
\par
Having written
$$
N_k=\multsum_{p_1,\ldots p_k\leq A^n}
\tilde{\expect}\Bigl(
\prod_{j=1}^k{X_{p_j}}
\Bigr)
=\multsum_{p_1,\ldots p_k\leq A^n}
\expect\Bigl(
\prod_{j=1}^k{X_{p_j}}
\Bigr)+O(A^{nk}c^{-n}),
$$
we can apply the equidistribution~(\ref{eq-quant-equid}) to each
expectation term, obtaining a main term which we will discuss in a
moment and a total error term $E$ which is bounded by
$$
E\ll A^{nk(1+B)}\rho^n+A^{nk}c^{-n}
$$
(where $B$ is as in~(\ref{eq-not-too-large})). Therefore $E$ tends to
$0$ as $n\ra +\infty$ if $A$ is chosen small enough (in terms of $k$),
which we assume to be done.
\par
There remains the main term. However, the latter is, by the
Independence Assumption~\ref{hyp-indep} and by retracing our steps,
almost tautologically the same as
$$
\expect\Bigl(\Bigl(\sum_{p\leq A^n}Y_p-\delta_p\Bigr)^k\Bigr)
$$
where the $(Y_p)$ are independent Bernoulli random variables with
expectation $\expect(Y_p)=\delta_p=|\Omega_p|/|\Gamma_p|$. It is a
basic probabilistic fact that the sum
$$
\sum_{p\leq A^n}Y_p
$$
satisfies the Central Limit Theorem, with mean $\kappa \log n$ and
variance $\kappa\log n$ (because of~(\ref{eq-dimension})
again). Therefore this sum has the right $k$-th moment for all $k\geq
0$, and this easily concludes the proof (or see~\cite{granville-sound}
for a direct analysis of this type of main terms to see the
combinatorics from which the normal moments explicitly appear).
% \par
% behaves like the $k$-th moment of a normal variable with mean
% $$
% \sum_{p\leq A^n}\frac{|\Omega_p|}{|\Gamma_p|}\sim \kappa\log n
% $$
% and variance
% $$
% \sum_{p\leq A^n}\frac{|\Omega_p|}{|\Gamma_p|}\Bigl(1-
% \frac{|\Omega_p|}{|\Gamma_p|}\Bigr)\sim\kappa \log n.
% $$
% \par
% If we grant this, we see easily that we can conclude, since we have
% $$
% 0\leq \omega(f(\gamma_n))-\sum_{\stacksum{p\leq
%     A^n}{\pi_p(\gamma_n)\in\Omega_p}}{1}
% \leq \frac{\log C}{\log A}
% $$
% where $C>1$ is such that $|f(g)|\leq C^{\ell_S(g)}$, unless
% $f(\gamma_n)=0$, which happens with exponentially small probability. 
% \par
% Define
% $$
% X_p=\charfun_{\pi_p(\gamma_n)\in\Omega_p}-\frac{|\Omega_p|}{|\Gamma_p|}.
% $$
% \par
% If we expand the $k$-th power defining $M_k$, we obtain
% $$
% M_k=\multsum_{p_1,\ldots p_k\leq A^n}
% \expect\Bigl(
% \prod_{j=1}^k{X_{p_j}}
% \Bigr). 
% $$
\end{proof}

\begin{example}[Erdös-Kac theorem for random $3$-manifolds]
  It is clear that the argument can be applied in greater generality
  (including for other counting methods, provided the analogue of
  quantitative and suitably uniform equidistribution is known). For
  instance, one sees that, for Dunfield--Thurston random
  $3$-manifolds, the number $\omega(\mathcal{M}_{\phi_n})$ of primes
  $p$ such that $H_1(\mathcal{M}_{\phi_n},\Fp_p)\not=0$ is such that
$$
\frac{\omega(\mathcal{M}_{\phi_n})-\log n}{\sqrt{\log n}}
$$
converges to a standard normal random variable, with the convention
$\omega(\mathcal{M}_{\phi_n})=0$ if
$H_1(\mathcal{M}_{\phi_n},\Qq)\not=0$. 
\end{example}

\section{The large sieve}\label{sec-large}

We begin with a motivating example.

\begin{example}
  Consider Corollary~\ref{cor-zariski-dense}. Although the Zariski
  topology contains a fair amount of information (see~\cite{bgs} for
  examples of distinction it makes concerning the sieve in orbits), it
  is not very arithmetic. By itself, the fact that $\Oc_f(\omega)$ is
  Zariski-dense in $\SL_r$ does not exclude the possibility that this
  set is contained, for instance, in the subset $X$ of $\SL_r(\Zz)$ of
  matrices where the top-left coefficient is a perfect square (since
  $X$ is Zariski-dense in $\SL_r$.) It is natural to try to study this
  and similar possibilities. The following definition is relevant
  (see~\cite[Chapter 3]{serre-galois}):

\begin{definition}[Thin set]
A subset $X\subset \SL_r(\Qq)$ is \emph{thin} if there exists an
algebraic variety $W/\Qq$ with $\dim(W)\leq r^2-1$ and a morphism
$W\fleche{\pi} \SL_r$ such that (1) $\pi$ has no rational section; (2)
we have $X\subset \pi(W(\Qq))$.
\end{definition}

\begin{example}
(1) The set $X=\{g\in \SL_r(\Qq)\,\mid\, g_{1,1}\text{ is a square}\}$ is
thin. Indeed, we have a $\Qq$-morphism
$$
\pi\,:\, \Aa^{r^2}\lra \Aa^{r^2}
$$
mapping $(g_{i,j})$ to the matrix $(h_{i,j})$ with $h_{1,1}=g_{1,1}^2$
and all other coordinates unchanged. The pull-back of this morphism to
$\SL_r\subset \Aa^{r^2}$ gives a morphism
$$
\pi\,:\, W\lra \SL_r,
$$
where
$$
W=\{g\in \Aa^{r^2}\,\mid\, \det(\pi(g))=1\},
$$
for which we have $X\subset \pi(W(\Qq))$ by construction (and $\dim
W\leq \dim \SL_r$ is clear since $\pi$ has finite fibers.)
\par
(2) A subset $X$ which is not Zariski-dense is thin.
\end{example}

We wish to prove:

\begin{proposition}\label{pr-not-thin}
Let $\Gamma$ and $f$ be as in Corollary~\ref{cor-zariski-dense}. Then
there exists $\omega\geq 1$ such that $\Oc_f(\omega)$ is not thin in
$\SL_r$. 
\end{proposition}

The natural idea to prove this is to prove that if $X$ is a thin set,
then for a random walk on $\Gamma$, the probability
$$
\proba(\gamma_n\in X)
$$ 
is too small to be compatible with~(\ref{eq-affine-sieve}). For this,
we observe, as in the proof of the Zariski-density, that if $X\subset
\pi(W(\Qq))$ for some
$$
\pi\,:\, W\lra \SL_r
$$
as in the definition, we have
$$
\pi_p(X)\subset \pi(W(\Fp_p)),
$$
for all primes $p$ large enough (such that $W$ and $\pi$ make sense
modulo $p$). Hence if $g\in X$, we have
$$
\pi_p(g)\notin \Omega_p=\SL_r(\Fp_p)-\pi(W(\Fp_p)),
$$
for all $p$ large enough.  This implies a sieve upper bound
$$
X\subset \sifted(\mathcal{P};\Omega),
$$
where $\mathcal{P}$ contains all but finitely many primes.
\par
However, the size of $\Omega_p$ is typically much larger than the
number of points of an algebraic variety, as one can guess by just
looking at the example of squares in $\Qq$, where the image modulo $p$
contains roughly half of all residue classes. Indeed, in general we have:

\begin{lemma}
  Let $\pi\,:\, W\lra \SL_r$ be a $\Qq$-rational morphism with $\dim
  W\leq r^2-1$ and with no $\Qq$-rational section. There exists
  $\delta<1$ such that, for $p$ large enough, we have
$$
\frac{|\pi(W(\Fp_p))|}{|\SL_r(\Fp_p)|}\leq \delta.
$$
\end{lemma}

For the proof, see e.g.~\cite[Th. 3.6.2]{serre-galois}. 
\end{example}

\begin{example}[Homology of Dunfield--Thurston random manifolds]\label{ex-h1}
  We consider the situation of Example~\ref{ex-first}, (3). Here we
  found sifting conditions $\Omega_p$ defined in~(\ref{eq-omega-h1})
  such that, if $\mathcal{M}_{\phi_n}$ denotes the manifold obtained
  from the $n$-th step of a random walk on the mapping class group
  $\Gamma$ (as in Example~\ref{ex-dt-1}), we have
$$
\proba(H_1(\mathcal{M}_{\phi_n},\Qq)\not=0)\leq
\proba(\phi_n\in \sifted(Q,\Omega_p))
$$
for any $Q\geq 1$, where $Q$ refers to using all primes $p\leq Q$ as
sieve conditions. It is an interesting computation to show that
$$
\frac{|\Omega_p|}{|\Gamma_p|}= 1-\prod_{j=1}^g{ \frac{1}{1+p^{-j}}}
$$
(see~\cite[Th. 8.4]{dunfield-thurston}) so that, for fixed $g$, there
exists $\delta_g>0$ for which
$$
\frac{|\Omega_p|}{|\Gamma_p|}\geq\delta_g
$$
for all $p$.
\end{example}

We now revert to the general setting of a discrete group $\Gamma$ with
local information $\pi_p\,:\, \Gamma\lra \Gamma_p$. We have found
above two natural instances of \emph{large sieves}, a terminology
which refers to sieving problems where the sets $\Omega_p$ are
``large'', something which most commonly means that they contain a
positive proportion of $\Gamma_p$: for some $\delta>0$, we have
\begin{equation}\label{eq-large}
\frac{|\Omega_p|}{|\Gamma_p|}\geq \delta>0
\end{equation}
for all $p\in \mathcal{P}$. This is to be compared with the ``small''
sieve assumption~(\ref{eq-dimension}), and this leads to an
interesting remark (answering the question to the reader at the end of
Example~\ref{ex-first}, (3)): the primes occur explicitly on both
sides of~(\ref{eq-dimension}), but as far as the left-hand side is
concerned, they are just indices that could be replaced with any other
countable set. However, on the right-hand side, the actual size of
primes (and hence their distribution) is involved. This feature
disappears in~(\ref{eq-large}). This suggests that the ``large'' sieve
could be of interest in much wider contexts outside of number
theory. This is indeed the case, as was shown already partly in the
book~\cite{cup}, and even more convincingly in the recent works of
Lubotzky, Meiri and Rosenzweig that we will discuss, some of which
prove general algebraic statements about linear groups using some
forms of sieve methods.
\par
To present the large sieve in the context of discrete groups, we will
use here the very simple version from the
paper~\cite{lubotzky-meiri-2} of Lubotzky and Meiri, adapted to our
setting.

\begin{theorem}[Large sieve]\label{th-ls}
  Let $\Gamma$ be a group generated by a finite symmetric set $S$ with
  $1\in S$. Let $\Gamma\lra \Gamma_p$ be surjective homomorphisms onto
  finite groups for $p\geq p_0$. Assume that:
\par
\emph{(1)} For any $p\not=q$ primes $\geq p_0$, the induced
homomorphisms
\begin{equation}\label{eq-indep-pair}
\Gamma\lra \Gamma_p\times \Gamma_q=\Gamma_{p,q}
\end{equation}
are onto.
\par
\emph{(2)} The family of Cayley graphs of $\Gamma_{p,q}$ and
$\Gamma_p$ with respect to $S$ is an expander family, for $p, q\geq
p_0$. 
\par
\emph{(3)} For some $B\geq 1$ we have
$$
|\Gamma_p|\leq p^B.
$$
\par
Let $\Omega_p\subset \Gamma_p$ be such that 
\begin{equation}\label{eq-large-hyp}
\frac{|\Omega_p|}{|\Gamma_p|}\geq \delta
\end{equation}
for some $\delta>0$ independent of $p$.
\par
Then there exists $A>1$ and $c>1$ such that for $Q=A^n$, we have
$$
\proba(\gamma_n\in \sifted(Q;\Omega))\ll c^{-n}
$$
for $n$ large enough, where the sieving is done using primes $p_0\leq
p\leq Q$.
\end{theorem}

Note how the assumptions concerning the group and the $\Gamma_p$ are
slightly weaker versions of those used for the small sieve in
Theorem~\ref{th-small}, since expansion and independence is only
required for pairs of primes instead of all squarefree integers. Thus
this version of the large sieve applies whenever
Theorem~\ref{th-small} is applicable. 
\par
In particular, in view of the example at the beginning of this
section, we see that this theorem proves
Proposition~\ref{pr-not-thin}. Similarly, for the Dunfield--Thurston
random manifolds of Example~\ref{ex-h1}, this implies the following:

\begin{proposition}
  Let $g\geq 1$ be an integer, and let $(\phi_n)$ be a random walk on
  the mapping class group $\Gamma$ of $\Sigma_g$ associated to a
  finite generating set $S$. Then there exists $c>1$ such that 
$$
\proba(H_1(\mathcal{M}_{\phi_n},\Qq)\not=0)\ll c^{-n}
$$
for $n\geq 1$. 
\end{proposition}

The fact that the probability tends to zero was already proved by
Dunfield and Thurston (see~\cite[Cor. 8.5]{dunfield-thurston}), and
the exponential decay was obtained in~\cite[Prop. 7.19]{cup}.

\begin{remark}
It would be unreasonable to expect lower-bounds for the size of
sifted sets in the large sieve situation, unless the set $\mathcal{P}$
determining the sieving conditions is extremely small (so the
situation essentially reverts to a bounded
sieve~(\ref{eq-bounded-sieve})). Indeed, if we consider integers and
sieve by removing the non-square residue classes modulo $p$ for all
$p\in\mathcal{P}$, which is certainly a large sieve, the right-hand
side of the heuristic size of the remaining set is
$(1/2)^{|\mathcal{P}|}$. If $\mathcal{P}$ is the set of primes $\leq
Q$, then this is much smaller than the number of squares in
$\{1,\ldots, N\}$, which certainly remain after the sieve, if
$Q=N^{\eps}$ for any fixed $\eps>0$. (See~\cite{green} for a
discussion of the fascinating question of the possibility of an
``inverse'' large sieve statement for integers.)
\end{remark}

We adapt the simple proof of Theorem~\ref{th-ls}
in~\cite{lubotzky-meiri-2} (due to R. Peled; it is reminiscent of some
classical arguments going back to Rényi and Tur\'an,
see~\cite[Prop. 2.15]{cup}.)

\begin{proof}
  For a fixed $n$, let $X_p$ denote the Bernoulli random variable
  equal to $1$ when $\pi_p(\gamma_n)\in \Omega_p$ and $0$ otherwise,
  and let
$$
X=\sum_{p_0\leq p\leq Q} X_p.
$$
\par
We see that $\gamma_n\in \sifted(Q;\Omega)$ is tautologically
equivalent to the condition $X=0$. We can compute easily the
expectation of $X$, namely
$$
\expect(X)=\sum_p \proba(\pi_p(\gamma_n)\in \Omega_p),
$$
which, by Expansion for $(\Gamma_p)$, satisfies
$$
\expect(X)=\sum_p\frac{|\Omega_p|}{|\Gamma_p|}
+O(Q^{1+B}\rho^n),
$$
where $\rho<1$ is an upper-bound for the spectral radius of the
expansion of the Cayley graphs. If $Q^{1+B}\ll \rho^n$, this gives
$$
\expect(X)\gg \pi(Q)\gg \frac{Q}{\log Q}
$$
using the large sieve assumption on the size of $\Omega_p$.
\par
We will now use the Chebychev inequality
$$
\proba(\gamma_n\in\sifted(Q;\Omega))=\proba(X=0)\leq 
\frac{\variance(X)}{\expect(X)^2},
$$
where $\variance(X)$ is the variance of $X$. We compute
$$
\variance(X)=\expect((X-\expect(X))^2)=\sum_{p,q} W(p,q)
$$
by expanding the square, where
\begin{align*}
W(p,q)&=\expect(X_pX_q)-\expect(X_p)\expect(X_q)
\\&=
\proba(\pi_p(\gamma_n)\in \Omega_p\text{ and }
\pi_q(\gamma_n)\in \Omega_q)-\proba(\pi_p(\gamma_n)\in
\Omega_p)\proba(\pi_q(\gamma_n)\in \Omega_q)
\end{align*}
(a measure of the correlation between two primes). We isolate the
diagonal terms where $p=q$, for which we use the trivial bound
$|W(p,p)|\leq 1$, and obtain
$$
\variance(X)\leq Q+Q^2\max_{p\not=q} |W(p,q)|.
$$
\par
Finally, to estimate $W(p,q)$ when $p\not=q$, we can apply the
assumption~(\ref{eq-indep-pair}) and the expansion of the Cayley
graphs of $\Gamma_{p,q}$ and $\Gamma_p$: we have
$$
\proba(\pi_p(\gamma_n)\in \Omega_p\text{ and }
\pi_q(\gamma_n)\in
\Omega_q)=\frac{|\Omega_p||\Omega_q|}{|\Gamma_p||\Gamma_q|}
+O(Q^{2B}\rho^n)
$$
while, by the same argument used for computing the expectation, we
have
$$
\proba(\pi_p(\gamma_n)\in \Omega_p)\proba(\pi_q(\gamma_n)\in
\Omega_q)= \frac{|\Omega_p||\Omega_q|}{|\Gamma_p||\Gamma_q|}
+O(Q^B\rho^n).
$$
\par
The main terms cancel, and therefore
$$
Q^2\max_{p\not=q} |W(p,q)|\ll Q^{2+2B}\rho^n.
$$
\par
Take $Q$ as large as possible so that $Q^{2+2B}\rho^n<1$, so that
$Q\geq A^n$ for some $A>1$. Then the Chebychev inequality gives
$$
\proba(\gamma_n\in\sifted(Q;\Omega))
\ll \frac{(Q+Q^{2+2B}\rho^n)(\log Q)^2}{Q^2}
\ll \frac{(\log Q)^2}{Q}
$$
which is of exponential decay in terms of $n$.
\end{proof}

\begin{remark}
  (1) Clearly, one can restrict the large sieve
  assumption~(\ref{eq-large-hyp}) to a subset of primes with positive
  natural density (e.g., some arithmetic progression) without changing
  the conclusion, and this is often useful.
\par
(2)   This very simple proof is well-suited to situations where precise
  information on the expansion constant of the relevant Cayley graphs
  is missing (as is most often the case). When such information is
  available, one gets from this argument an explicit constant $c>1$,
  and one may wish to get it as large as possible. For this, one can
  use rather more precise inequalities, as discussed extensively
  in~\cite{cup}.
\end{remark}

The point of the large sieve is really the exponential decay it
provides. If one is interested in a statement of qualitative decay
\begin{equation}\label{eq-qual-large}
\lim_{n\ra +\infty} \proba(\gamma_n\in X)=0
\end{equation}
for a subset $X\subset \Gamma$ such that
$$
X\subset \sifted(Q;\Omega)
$$
for \emph{all} $Q$ large enough, where the $\Omega_p$
satisfy~(\ref{eq-large-hyp}), then one can more easily apply the
bounded sieve~(\ref{eq-bounded-sieve}) to a finite set $I$, getting
$$
\limsup_{n\ra +\infty} \proba(\gamma_n\in X)\leq 
\lim_{n\ra +\infty} \proba(\pi_p(\gamma_n)\notin \Omega_p\text{ for }
p\in I)=
\prod_{p\in I}\Bigl(1-\frac{|\Omega_p|}{|\Gamma_p|}\Bigr)
\leq (1-\delta)^{|I|}.
$$
\par
Then, letting $|I|\ra +\infty$, we obtain~(\ref{eq-qual-large}). As an
example, note that this qualitative decay is \emph{not} sufficient to
prove Proposition~\ref{pr-not-thin}.
\par
Lubotzky and Meiri introduce the following convenient definition:

\begin{definition}[Exponentially small sets]
Let $\Gamma$ be a finitely generated group. A subset $X\subset \Gamma$
is \emph{exponentially small} if, for any finite symmetric generating
set $S$ containing $1$, and with $(\gamma_n)$ the corresponding random
walk on $\Gamma$, there exists a constant $c_S>1$ such that
$$
\proba(\gamma_n\in X)\ll c_S^{-n}
$$
for $n\geq 1$.
\end{definition}

\begin{remark}
  Thus, we can summarize part of our previous discussion by stating
  that if $X$ is a thin subset of $\SL_r(\Qq)$, and $\Gamma$ is a
  finitely generated Zariski-dense subgroup of $\SL_r(\Zz)$, then
  $X\cap \Gamma$ is exponentially small in $\Gamma$, and by saying
  that the set of mapping classes (in a fixed mapping class group
  $\Gamma$ of genus $g\geq 1$) for which the corresponding manifold
  obtained by Heegaard splitting has positive first rational Betti
  number is exponentially small.
\end{remark}

The first inkling of the large sieve in non-abelian discrete groups is
found in applications of the qualitative argument above by
Dunfield--Thurston~\cite{dunfield-thurston} and Rivin~\cite{rivin,
  rivin2} in geometric contexts (the second paper~\cite{rivin2} of
Rivin was the first to obtain exponential decay, though its
publication was delayed by a journal with overly long backlog; we
thank I. Rivin for clarifying the priority in publication here). We
illustrate further the large sieve with an example from the second,
and then discuss briefly two other applications.

\begin{example}[Pseudo-Anosov elements of the mapping class group]
  Let $g\geq 1$ be given and let $\Gamma$ be the mapping class group
  of $\Sigma_g$. Thurston's celebrated theory classifies the elements
  $\gamma\in \Gamma$ as (1) reducible; (2) finite-order; or (3)
  pseudo-Anosov. To quantify the feeling that ``most'' elements are of
  the third type, Rivin used a criterion based on the action of
  $\gamma$ on $H_1(\Sigma_g,\Zz)$, which says that \emph{if} (but not
  only if) the characteristic polynomial $P_{\gamma}$ of this action
  is $P_{\gamma}$ is irreducible, and satisfies further easy
  conditions, then $\gamma$ is pseudo-Anosov. One then notes that if
  $P_{\gamma}$ is reducible, then so is its reduction modulo any
  prime, so $\pi_p(\gamma)$ is not in the subset $\Omega_p$ of
  elements of $\Sp_{2g}(\Fp_p)$ for which the characteristic
  polynomial is irreducible. A computation that goes back to
  Chavdarov~\cite[\S 3]{chavdarov} shows that, for some $\delta>0$, we
  have
$$
\frac{|\Omega_p|}{|\Sp_{2g}(\Fp_p)|}\geq \delta>0
$$
for all $p$, and hence the large sieve applies. A simple further
argument deals with the other necessary conditions in the
pseudo-Anosov criterion, and one concludes that the set of
non-pseudo-Anosov elements is exponentially small in $\Gamma$. 
\par
It should be said, however, that this proof is to some extent
unsatisfactory, because it doesn't use the deeper structural and
dynamical properties of pseudo-Anosov elements. For instance, using
the action on homology means that one cannot argue similarly for
subgroups $\tilde{\Gamma}\subset \Gamma$ for which the action on
homology is small, especially subgroups of the Torelli group, which is
defined precisely as the kernel of the homomorphism
$$
\Gamma\lra \Sp_{2g}(\Zz)
$$
giving this action.
\par
However, Maher~\cite{maher-3, maher2} has shown, using more geometric
methods, that non-pseudo-Anosov elements are exponentially small in
any subgroup of $\Gamma$, except those for which this property is
false for obvious reasons, and his work applies in particular to the
Torelli subgroup.
\par
On the other hand, Lubotzky--Meiri~\cite{lubotzky-meiri} and
Malestein--Souto~\cite{malestein-souto} (independently) have recently
found proofs that non-pseudo-Anosov elements are exponentially small
in the Torelli group using ideas similar to those above.
\end{example}

\begin{example}[Powers in linear groups]
In~\cite{lubotzky-meiri-2}, Lubotzky and Meiri prove the following
statement using the large sieve. The reader should note that this is,
on the face of it, a purely  algebraic property  of finitely generated
linear groups.

\begin{theorem}[Lubotzky--Meiri]
  Let $\Gamma$ be a finitely generated subgroup of $\GL_r(\Cc)$ for
  some $r\geq 2$. If $\Gamma$ is not virtually solvable,\footnote{\
    I.e., there is no finite-index solvable subgroup of $\Gamma$.}
  then the set $X$ of proper powers, i.e., the set of those
  $g\in\Gamma$ such that there exists $k\geq 2$ and $h\in \Gamma$ with
  $g=h^k$, is exponentially small in $\Gamma$.
\end{theorem}

This strenghtens considerably some earlier work of a Hrushovski,
Kropholler, Lubotzky and Shalev~\cite{hkls}. The proof is also very
instructive, in particular by showing how sieve should be considered
as a \emph{tool} among others: here, one can use the large sieve to
control elements which are $k$-th powers for a \emph{fixed} $k\geq 2$,
but taking the union over all $k\geq 2$ cannot be done with sieve
alone. So Lubotzky and Meiri use other tools to deal with large values
of $k$, in that case based on ideas related to the work of Lubotzky,
Mozes and Raghunathan comparing archimedean and word-length
metrics~\cite{lmr}.

\end{example}

\begin{example}[Typical Galois groups of characteristic polynomials]
  Our last example has been studied by Rivin~\cite{rivin},
  Jouve--Kowalski--Zywina~\cite{jkz},
  Gorodnik--Nevo~\cite{gorodnik-nevo} and most recently
  Lubotzky--Rosenzweig~\cite{lubotzky-rosenzweig}, who were the first
  to explicitly consider the case of sparse subgroups. However, the
  underlying idea of probabilistic Galois theory goes back to versions
  of Hilbert's irreducibility theorem, and especially to Gallagher's
  introduction of the large sieve in this
  context~\cite{gallagher-2}. (There are also relations with works of
  Prasad and Rapinchuk~\cite{pr, pr2}.)
\par
In the (most general) version of Lubotzky--Rosenzweig, one considers a
finitely generated field $K\subset \Cc$ and a finitely generated
subgroup $\Gamma\subset \GL_r(K)$ for some $r\geq 2$. The basic
question is: what is the ``typical'' behavior of the splitting field
of the characteristic polynomial $\det(T-g)\in K[T]$ for some element
$g\in \Gamma$?
\par
This can be studied using the large sieve, as we explain in the
simplest case when $\Gamma\subset \SL_r(\Zz)$. Let $\Gg$ be the
Zariski-closure of $\Gamma$, and assume that $\Gg$ is connected and
split over $\Qq$, for instance $\Gg=\SL_r$. Let $W$ be the Weyl group
of $\Gg$: this will turn out to be the typical Galois group in this
case.
\par
To see this, the first ingredient is the existence, for any prime $p$
large enough (such that $\Gg$ can be reduced modulo $p$), of a certain
map
$$
\varphi\,:\, \Gg(\Fp_p)_{r}^{\sharp}\lra W^{\sharp}
$$
going back to Carter and Steinberg, where $G^{\sharp}$ denotes the set
of conjugacy classes of a finite group and the subscript $r$ restricts
to regular semisimple elements in the finite group $\Gg(\Fp_p)$. 
\par
This map is used to detect elements in the Galois groups of elements
in $\Gamma$ in the following way. First, for $g\in \SL_r(\Zz)$, let
$P_g$ be the characteristic polynomial and let $K_g$ be its splitting
field, $\Gal_g$ its Galois group over $\Qq$.  The point is that, if
$g$ is a regular semisimple element of $\Gamma$, it is shown
in~\cite{jkz} that there exists an injective homomorphism
$$
j_g\,:\, \Gal_g\injecte W,
$$
canonical up to conjugation, such that if $p$ is any prime unramified
in $K_g$, the Frobenius conjugacy class at $p$ maps under $j_g$ to the
conjugacy class $\varphi(\pi_p(g))\in W^{\sharp}$. Thus one can detect
whether the image of $\Gal_g$ in $W$ intersects various conjugacy
classes by seeing where the reduction modulo $p$ of $g$ lies with
respect to $\varphi$. As it turns out, the image of $\varphi$ becomes
equidistributed among the conjugacy classes in $W$ as $p$ becomes
large. Using this, it is not too hard to show that if $\alpha\in
W^{\sharp}$ is a given conjugacy class and if $\Omega_p$ denotes the
set of $g\in\Gg(\Fp_p)$ such that $\varphi(g)\notin \alpha$, then
these sets satisfy a large sieve density assumption
$$
\frac{|\Omega_g|}{|\Gg(\Fp_p)|}\geq \delta_{\alpha}>0
$$
for some $\delta_{\alpha}>0$ and all $p$ large enough. It follows by
the large sieve that the probability that the element $\gamma_n$ at
the $n$-th step of a random walk on $\Gamma$ has Galois group such
that $j_g(\Gal_g)\cap \alpha=\emptyset$ is exponentially small. This
holds for all the finitely many classes in $W$, and a well-known lemma
of Jordan\footnote{\ In a finite group $G$, there is no proper
  subgroup $H$ such that $H\cap \alpha\not=\emptyset$ for all
  conjugacy classes $\alpha$ in $G$.} allows us to conclude that the
set of $g\in \Gamma$ where $j_g$ is not onto is exponentially small.
\par
The general case treated by Lubotzky--Rosenzweig is quite a bit more
involved. In particular, new phenomena appear when $\Gg$ is \emph{not}
connected, and the different cosets of the connected component of the
identity then usually have different typical Galois groups. We refer
to their paper for details.
\end{example}

\section{Problems and questions}\label{sec-pbs}

We discuss here a few questions and problems, selected to a large
extent according to the author's own interests and bias.

\begin{enumerate}
\item \ [Effective results] A striking aspect of the results we have
  described is how little they use the many refinements and
  developments of sieve theory, as described in~\cite{FI} for small
  sieves, and in~\cite{cup} for the large sieve. This is due to the
  almost complete absence of explicit forms of the Expansion
  Assumption for sparse groups, from which it follows that one can
  not, for instance, give a numerical value of the integer $\omega$
  guaranteed to exist in Corollary~\ref{cor-zariski-dense} (recall
  that in classical sieves, the current state of the art is very
  refined indeed: one knows, for instance, that the number of primes
  $p\leq x$ such that $p+2$ has at most two prime factors is of the
  right order of magnitude). In fact, when implementing the
  combinatorial counting methods (either word-length or random walks),
  there is \emph{no} known explicit sieve statement, as far as the
  author knows\footnote{\ The remarkable results of Bourgain and
    Kontorovich~\cite{bourgain-konto} are explicit, but not directly
    related to the sieve as we have considered here;
    see~\cite{konto-circle} for a survey in these Proceedings.}
  (whereas a few explicit bounds do exist for archimedean balls, based
  on spectral or ergodic methods, see, e.g., the works of
  Kontorovich~\cite{konto}, Kontorovich--Oh~\cite{k-o-2},
  Nevo--Sarnak~\cite{n-s}, Liu--Sarnak~\cite{liu-s}, and
  Gorodnik--Nevo~\cite{gn2}, or for random walks in a few arithmetic
  groups~\cite{cup}). It seems clear that the current proofs of
  expansion for sparse groups, although they are effective, would lead
  to dreadful bounds on a suitable $\omega$ (see~\cite{explicit} for a
  numerical upper-bound on the spectral radius for Cayley graphs of
  Zariski-dense subgroups of $\SL_2(\Zz)$ modulo primes, which
  suggests, e.g., that one could not get better than $\omega$ of size
  at least $2^{2^{40}}$ or so for the product of coordinates function
  on the Lubotzky group\ldots).
\item \ [Average expansion?]
One possibility suggested by the classical Bombieri--Vino\-gradov
Theorem is to attempt a proof of expansion ``on average'' for the
relevant Cayley graphs: for many applications, it would be sufficient
to prove estimates for quantities like
$$
\sum_{q_I\leq Q}\max_{(g_p)\in\Gamma_I} \Bigl|
\mu_N(\pi_p(g)=g_p\text{ for } p\in I)-
\frac{1}{|\Gamma_I|}\Bigr|,
$$
and such estimates could conceivably be provable without resorting to
individual estimates for each $q_I$. They could also, optimistically,
be of better quality than what is true for individual $I$. (Such a
property is known for classical sieve, by work of Fouvry, Bombieri,
Friedlander and Iwaniec). 
\item \ [Combinatorial balls] It would be very interesting to have
  equidistribution and sieve results using trunctions based on
  word-length balls, without resorting to random walks. Here, the hope
  is that one might not need to compute the asymptotics of the size of
  the combinatorial balls, since one is only interested in relative
  proportions of elements in a ball mapping to a given
  $g\in\Gamma_p$. 
\item \ [Reverse power] This question is related to (1): at least in
  some cases, one has very convincing conjectures for the counting
  function of primes arising from small sieve in orbits (see,
  e.g.,~\cite{fuchs, f-s}). Suppose one assumes such conjectures. What
  does this imply for prime numbers? In other words, can one exploit
  information on primes represented using the sieve in orbits to
  derive other properties of prime numbers? Here the reference to keep
  in mind is the result of Gallagher (see~\cite{gallagher} and the
  generalization in~\cite{poisson}) that shows that uniform versions
  of the Hardy--Littlewood $k$-tuples conjecture imply that the number
  of primes $p\leq x$ in intervals of length $\lambda \log x$, for
  fixed $\lambda>0$, is asymptotically Poisson-distributed.
\end{enumerate}

\end{document}